\documentclass[11pt]{article}

\usepackage{fullpage}
\usepackage{authblk}
\usepackage{natbib}
\usepackage{algorithm}
\usepackage{algpseudocode}
\usepackage{amsmath,amsthm,amsfonts}
\usepackage{amsmath}
\usepackage[subnum]{cases}
\usepackage{tcolorbox}
\usepackage{mathrsfs}
\usepackage{amssymb}
\usepackage{color}
\usepackage{mathrsfs}
\usepackage{enumitem}
\usepackage{bm}
\usepackage{multirow}
\usepackage{booktabs}
\usepackage{makecell}
\usepackage{graphicx}
\usepackage{subcaption}
\usepackage{comment}
\usepackage{cases}
\usepackage{appendix}
\usepackage{tikz}
\usetikzlibrary{arrows,shapes}
\usepackage[colorlinks= true, linkcolor=brown, citecolor=blue, urlcolor=black]{hyperref}
\usepackage{cleveref}
\usepackage{marginnote}
\usepackage{diagbox} 
\usepackage{empheq}

\numberwithin{equation}{section}

\newtheorem{theorem}{Theorem}[section]

\newtheorem{corollary}[theorem]{Corollary}

\usepackage{booktabs}
\usepackage{tikz}
\usepackage{pgfplots}
\usetikzlibrary{calc,intersections}

\usetikzlibrary{shapes.geometric, arrows, positioning}

\tikzset{
	mybox/.style  = {draw, rectangle, minimum width=6cm, minimum height=0.8cm, text centered, text width=6.4cm,   
		font=\normalsize},
	box/.style  = {draw, rectangle, minimum width=2.0cm, minimum height=0.6cm, text centered, text width=3.0cm,   
		font=\normalsize},
	myarrow/.style = {line width=0.2pt, draw=black, -triangle 60, postaction={draw, line width=0.2pt, shorten >=10pt,-}}
}

\tikzstyle{arrow} = [->, >=stealth, -triangle 60]

\allowdisplaybreaks

\makeatletter
\newcommand{\leqnomode}{\tagsleft@true}
\newcommand{\reqnomode}{\tagsleft@false}
\makeatother

\begin{document}

\title{Revisiting the acceleration phenomenon via high-resolution differential equations\thanks{This work was supported by Grant No.YSBR-034 of CAS and Grant No.12288201 of NSFC.}}


\author[ ]{Shuo Chen \qquad Bin Shi\thanks{Corresponding author, Email: \url{shibin@lsec.cc.ac.cn}} \qquad Ya-xiang Yuan} 
\affil[ ]{State Key Laboratory of Scientific and Engineering Computing, Academy of Mathematics and Systems Science, Chinese Academy of Sciences, Beijing 100190, China}
\affil{University of Chinese Academy of Sciences, Beijing 100049, China}

\date\today

\maketitle

\begin{abstract}
Nesterov's accelerated gradient descent (\texttt{NAG}) is one of the milestones in the history of first-order algorithms. It was not successfully uncovered until the high-resolution differential equation framework was proposed in~\citep{shi2021understanding} that the mechanism behind the acceleration phenomenon is due to the gradient correction term. To deepen our understanding of the high-resolution differential equation framework on the convergence rate, we continue to investigate~\texttt{NAG} for the $\mu$-strongly convex function based on the techniques of Lyapunov analysis and phase-space representation in this paper.\footnote{Throughout this paper,~\texttt{NAG} without any special emphasization is only for the $\mu$-strongly convex function, where the rigorous definition is shown in~\Cref{subsec: notation-organization}. } First, we revisit the proof from the gradient-correction scheme. Similar to~\citep{chen2022gradient}, the straightforward calculation simplifies the proof extremely and enlarges the step size to $s=1/L$ with minor modification. Meanwhile, the way of constructing Lyapunov functions is principled. Furthermore, we also investigate~\texttt{NAG} from the implicit-velocity scheme. Due to the difference in the velocity iterates, we find that the Lyapunov function is constructed from the implicit-velocity scheme without the additional term and the calculation of iterative difference becomes simpler. Together with the optimal step size obtained, the high-resolution differential equation framework from the implicit-velocity scheme of~\texttt{NAG} is perfect and outperforms the gradient-correction scheme.

%
%
%
%
%
%
%
%
\end{abstract}

\section{Introduction}
\label{sec: intro}

In modern statistical machine learning, a predominant characteristic is that data sets and parameter spaces are oversized. Designing and improving efficient large-scale optimization algorithms is paramount for assessing these issues. Practically, it is often infeasible to compute and store the Hessians. Hence, gradient-based optimization has become prevalent due to its relatively cheap computation and storage. Recall the smooth unconstrained optimization problem is formulated as
\[
\min_{x\in\mathbb{R}^{d}}~f(x).
\]
In this paper, we consider the objective focusing on the $\mu$-strongly convex function. Perhaps the simplest first-order method is the vanilla gradient descent, implemented recursively with fixed step size $s$ as
\[
x_{k+1} = x_{k} - s\nabla f(x_{k}),
\]
given any initial point $x_{0} \in \mathbb{R}^d$. A simple squared distance brings about linear convergence as
\[
f(x_{k}) - f(x^\ast) \leq O\left((1-\mu s)^{k}\right).
\]
The modern idea of acceleration takes its rise from~\citep{polyak1964some}, where his heavy-ball method (or momentum method) locally accelerates the linear convergence. Afterward, the milestone work is Nesterov's accelerated gradient descent (\texttt{NAG})
\[
\left\{
\begin{aligned}
&x_{k+1} = y_{k}-s\nabla f(y_{k}), \\
&y_{k+1} = x_{k+1} + \frac{1-\sqrt{\mu s}}{1+\sqrt{\mu s}} \cdot(x_{k+1} - x_{k}),
\end{aligned}
\right.
\]
with any initial $x_0 = y_0 \in \mathbb{R}^d$, which globally accelerates the linear convergence to
\[
f(x_{k}) - f(x^\ast) \leq O\left((1-\sqrt{\mu s})^{k}\right).
\]
The proof is provided by~\citep{nesterov1998introductory} via extending his technique of~\textit{Estimate Sequence} for the convex function in~\citep{nesterov1983method} to the $\mu$-strongly convex function.\footnote{It should be noticed that~\citep{nesterov1983method} originates~\texttt{NAG} only focusing on the convex function. }


However, Nesterov's technique of~\textit{Estimate Sequence} is so algebraically complex that the cause of acceleration is still unclear. Recently,~\citep{shi2021understanding} proposed the high-resolution differential equation framework accompanied by the corresponding techniques of phase-space representation and Lyapunov analysis, which successfully lifted the veil of the mysterious acceleration phenomenon through the discovery of gradient correction. The gradient correction is small but essential because it exists in~\texttt{NAG} but is not included in Polyak’s heavy-ball method, which leads to a key difference in the derivative inequalities with the additional $O(\sqrt{s})$-term to cancel out the discrete error and preserve the acceleration.

In~\citep{shi2021understanding}, we also investigate~\texttt{NAG} for the convex function based on the high-resolution differential equation framework and dig out that the squared gradient norm converges with an inverse cubic rate. The proof is simplified straightforwardly and concisely in~\citep{chen2022gradient}, where an equivalent proof from the implicit-velocity scheme is also shown. Meanwhile, the step size is enlarged to $s=1/L$. For the convex case, the Lyapunov function based on the high-resolution differential equation and the phase-space representation is tailor-made for the~\texttt{NAG}. Moreover, a key inequality of the gradient step is found in~\citep{li2022proximal}, which generalizes the gradient norm minimization to composite optimization. Spontaneously, we hope to carry the perfect framework to the $\mu$-strongly convex case. As the motivation, this provokes us to proceed with further investigation, rigorously formulated by two questions:

\begin{tcolorbox}
\centering
\begin{itemize}
\item \textbf{Question-1}: Can the proof of~\citep[Theorem 3]{shi2021understanding} be simplified straightforwardly and concisely? Can the step size be improved to $s=1/L$?
\item \textbf{Question-2}:  Can the accelerated convergence rate be obtained from the implicit-velocity scheme of~\texttt{NAG}?
\end{itemize}
\end{tcolorbox}

\subsection{A principled way of constructing Lyapunov functions}
\label{subsec: principle-lyapunov}

While querying the answers for the above two questions, we find that the way to construct Lyapunov functions in~\citep{shi2021understanding} is principled instead of blind patching together. Here, we refine it to make a highlight. The Lyapunov functions constructed in~\citep{shi2021understanding} are mainly composed of potential, kinetic, and mixed energy, with an additional term in the discrete case. Indeed, the ordering of calculating the time derivative (continuous case) or the iterative difference (discrete case), regardless of the coefficients, should consider first mixed energy, then turn to kinetic energy, and finally, the coefficient of potential energy is set to match two previous terms with an additional term to balance out residual terms. The detailed calculation is shown as follows. 

\begin{itemize}
\item[\textbf{(1)}] For the mixed energy, it should be considered first as a squared norm of some combining terms. The time derivative of 
some combining terms for the continuous high-resolution differential equation, or the iterative difference for the discrete algorithm, equals a term only including the gradient. 
\item[\textbf{(2)}] Then, for the sake of matching the corresponding terms proportional, we must consider the kinetic energy. Otherwise, the squared norm of velocity does not appear in the time derivative or iterative difference of Lyapunov functions, while the velocity is included in the mixed energy. 
\item[\textbf{(3)}] Finally, the coefficient of potential energy is set to cancel out the inner product of gradient and velocity from the time derivative (or iterative difference) of two previous terms. An additional term for the discrete algorithm is set to balance out residual terms.
\end{itemize}

For the convex function, we also construct a Lyapunov function in the same principled way except for the kinetic energy in~\citep{shi2021understanding, chen2022gradient, li2022proximal}. This is because it only requires that the time derivative (or iterative difference) of the Lyapunov functions is no more than zero instead of matching the corresponding terms proportional.

\subsection{Main contributions}
\label{subsec: overview-contributions}
In this paper, we continue to study the convergence rate of \texttt{NAG} based on the high-resolution differential equation framework. Recall the Lyapunov function's coefficients in~\citep{shi2021understanding} include $1-\sqrt{\mu s}$ as the numerator, and then some coefficients probably degenerate to zero when the step size is set $s=1/L$ and the conditional number $\nu/L=1$. This leads to zero probably appearing in the denominator while matching the corresponding terms proportional, which does not make sense. To avoid the situation, we change the momentum coefficient from $\frac{1-\sqrt{\mu s}}{1 + \sqrt{\mu s}}$ to $\frac{1}{1+2\sqrt{\mu s}}$ and write it down as
\begin{equation}
\label{eqn: nag}
\left\{
\begin{aligned}
&x_{k+1} = y_{k}-s\nabla f(y_{k}), \\
&y_{k+1} = x_{k+1} + \frac{x_{k+1} - x_{k}}{1+2\sqrt{\mu s}} ,
\end{aligned}
\right.
\end{equation}
with any initial $x_0=y_0 \in \mathbb{R}^d$. Actually in the sense of asymptotic expansion, both the coefficients, $\frac{1}{1+2\sqrt{\mu s}}$ and $\frac{1-\sqrt{\mu s}}{1 + \sqrt{\mu s}}$, are equivalent. In other words, the modified scheme~\eqref{eqn: nag} is not essentially different from~\texttt{NAG}, which still keeps the first gradient step and the second momentum step. For convenience, we use the modified scheme~\eqref{eqn: nag} to note~\texttt{NAG} in the following part of this paper. Moreover, with~\eqref{eqn: nag}, we show that the proofs on the operation of complex coefficients are simplified extremely in~\Cref{sec: gradient-correction} and~\Cref{sec: implicit}. 

\paragraph{}In order to answer the two questions above, we make the following contributions:  

\paragraph{Proof from the gradient-correction scheme} The gradient-correction scheme of~\texttt{NAG} is rewritten down with the single sequence $\{y_k\}_{k=0}^{\infty}$ as
\begin{equation}
\label{eqn: nag-y}
y_{k+1} = y_{k} +  \frac{y_k - y_{k-1}}{1 + 2\sqrt{\mu s}} - s\cdot\nabla f(y_k) - s \cdot \frac{ \overbrace{\nabla f(y_{k}) - \nabla f(y_{k-1})}^{\textbf{gradient correction}} }{1 + 2\sqrt{\mu s}}.
\end{equation}
We revisit the proof shown in~\citep[Section 3]{shi2021understanding} and give a straightforward proof for the gradient-correction scheme~\eqref{eqn: nag-y} with a little modification, where the observation is consistent with the convex case~\citep{chen2022gradient}. An inequality for the first gradient step of~\texttt{NAG}~\eqref{eqn: nag} can be derived as
\begin{equation}
\label{eqn: gradient-inequality}
f(x_{k+1}) - f(x^\star) \leq f(y_k) - f(x^\star) - \frac{s}{2} \|\nabla f(y_k)\|^2.
\end{equation}
for any step size $0 < s\leq 1/L$. Together with the gradient inequality~\eqref{eqn: gradient-inequality}, we enlarge the step size to $s = 1/L$. Moreover, the discrete convergence rate is in full accord with the continuous high-resolution differential equation.

\paragraph{Proof from the implicit-velocity scheme} The implicit-velocity scheme of~\texttt{NAG} is rewritten down with the single sequence $\{x_k\}_{k=0}^{\infty}$ as 
\begin{equation}
\label{eqn: nag-x}
  x_{k+1} = x_k + \frac{x_k - x_{k-1}}{1 + 2\sqrt{\mu s}} - s\nabla f\bigg( x_k + \frac{\overbrace{x_k - x_{k-1}}^{\textbf{implicit velocity}}}{1 + 2\sqrt{\mu s}} \bigg).
\end{equation}
Similar to the convex case~\citep{chen2022gradient}, we start to consider the implicit-velocity scheme~\eqref{eqn: nag-x}. A corresponding implicit-velocity high-resolution differential equation with its continuous convergence rate is derived as guidance. And then, we find that the discrete Lyapunov function constructed from the implicit-velocity scheme~\eqref{eqn: nag-x} does not need to include the additional term. Moreover, without modifying the coefficients like in the case of the gradient-correction scheme~\eqref{eqn: nag-y}, we obtain the identically accelerated convergence rate and enlarge the step size to $s=1/L$. This is essentially due to the difference of kinetic energy in the Lyapunov function with both potential and mixed energies identical. More specifically, the sequence of velocity $\{v_k\}_{k=0}^{\infty}$ is generated from the sequence of $\{x_k\}_{k=0}^{\infty}$ for the implicit-velocity scheme~\eqref{eqn: nag-x} while from the sequence of $\{y_k\}_{k=0}^{\infty}$ for the gradient-correction scheme~\eqref{eqn: nag-y}. In addition, another advantage of the implicit-velocity scheme~\eqref{eqn: nag-x} is only three terms in the velocity iterates of phase-space representation without explicitly including the gradient-correction term, a difference of two adjacent gradients, which extremely reduces the computation of kinetic energy's iterative difference.

%
%
%
%
%
%


\paragraph{} In summary, we find that the implicit-velocity scheme is perfect and superior to the gradient-correction scheme with a simple comparison with two Lyapunov functions. Combined with the convex case~\citep{chen2022gradient, li2022proximal}, we conclude that the high-resolution differential equation framework from the implicit-velocity scheme~\eqref{eqn: nag-x} accompanied by the techniques of phase-space representation and Lyapunov function is tailor-made for~\texttt{NAG}.

\subsection{Notations and organization}
\label{subsec: notation-organization}

In this paper, we follow the notations used in~\citep{shi2021understanding}. Let $\mathcal{S}_{\mu, L}^1(\mathbb{R}^d)$ be the class of $L$-smooth and $\mu$-strongly convex functions defined on $\mathbb{R}^d$; that is, $f \in \mathcal{S}_{\mu, L}^1$ if for any $ x,y\in\mathbb{R}^{d}$, its objective value satisfies $f(y)\geq f(x)+\langle\nabla f(x), y-x\rangle+\frac{\mu}{2}\|y-x\|^{2}$ and its gradient is $L$-Lipschitz continuous in the sense that $\|\nabla f(y) - \nabla f(x)\|\leq L\|y-x\|$, where $\|\cdot\|$ denotes the standard Euclidean norm and $\mu>0$ and $L > 0$ are the strong convex constant and the Lipschitz constant respectively. Besides, $x^\star$ denotes the minimizer of the convex function $f$. 

The remainder of this paper is organized as follows. In~\Cref{sec: related-work}, some related research works are described. A straightforward proof from the gradient-correction scheme is provided in~\Cref{sec: gradient-correction}. In~\Cref{sec: implicit}, we derive the convergence rate from the implicit-velocity scheme, which is guided by its corresponding high-resolution implicit-velocity differential equation.  In~\Cref{sec: discussion}, we briefly discuss the coexistence of acceleration and monotonicity and then conclude the paper with some future researches. 


\section{Related works}
\label{sec: related-work}

Within the first-order information, perhaps it dates from a two-step gradient method, the Ravine method in~\citep{gelfand1962some}, to improve the single-step gradient descent. Until in~\citep{nesterov1983method},~\texttt{NAG} for the convex function is not proposed to bring about the global acceleration phenomenon, which is generalized to the $\mu$-strongly convex function in~\citep{nesterov1998introductory}. As far as~\texttt{NAG} itself is concerned, the research papers on the $\mu$-strongly convex function are far less than the convex function. However, from the structure of algorithms,~\citep{nesterov2010recent} points out that~\texttt{NAG} for the $\mu$-strongly convex function and Polyak's heavy-ball method~\citep{polyak1964some} are close and easy to compare. And then, this view is emphasized further in~\citep{jordan2018dynamical}. Finally, the veil behind the acceleration phenomenon is lifted by the discovery of gradient correction on the basis of the high-resolution differential equation framework in~\citep{shi2021understanding}.

For the $\mu$-strongly convex function, the technique of the Lyapunov function is first used to analyze the acceleration phenomenon~\citep{wilson2021lyapunov}. Following the high-resolution differential equation framework~\citep{shi2021understanding}, several research works from the view of numerical analysis are proposed to investigate further in~\citep{luo2021differential, zhang2021revisiting}.  It is proposed by combining the low-resolution differential equation and the continuous limit of the Newton method to design and analyze new algorithms in~\citep{attouch2020first}, where the terminology is called inertial dynamics with a Hessian-driven term. Although the inertial dynamics with a
Hessian-driven term resembles closely with the high-resolution differential equations in~\citep{shi2021understanding}, it is important to note that the Hessian-driven terms are from the second-order information of Newton’s method~\citep{attouch2014dynamical}, while the gradient correction entirely relies on the first-order information of Nesterov's accelerated gradient descent method itself.  The implicit-velocity view from the continuous differential equation is first proposed in~\citep{muehlebach2019dynamical}. When the noise is added, the momentum scheme shows more power for the nonconvex function in~\citep{shi2021hyperparameters}. In addition, the momentum-based scheme is also generalized to residual neural networks~\citep{xia2021heavy}.

\section{A straightforward proof from the gradient-correction scheme}
\label{sec: gradient-correction}

In this section, we provide a straightforward proof from the gradient-correction scheme~\eqref{eqn: nag-y} with a little modification. Particularly in~\Cref{subsec: principle-lyapunov}, we show how to construct Lyapunov functions in the principled way. And then, the step size enlarged to $s = 1/L$ is shown in~\Cref{subsec: step-size-L}. Here, we ignore to demonstrate the convergence rate of the continuous high-resolution differential equation since there is no essential difference from~\citep[Theorem 1]{shi2021understanding} except for some modified coefficients.

Before proceeding to construct a Lyapunov function, we make a minor modification  for the gradient-correction scheme~\eqref{eqn: nag-y} as
\begin{equation}
\label{eqn: nag-y-modify}
y_{k+1} = y_{k} +  \frac{y_k - y_{k-1}}{1 + 2\sqrt{\mu s}} - s\cdot\frac{1}{\underbrace{1 + 2\sqrt{\mu s}}_{\textbf{modification}}}\nabla f(y_{k}) - s \cdot \frac{ \overbrace{\nabla f(y_{k}) - \nabla f(y_{k-1})}^{\textbf{gradient correction}} }{1 + 2\sqrt{\mu s}}.
\end{equation}
Taking the velocity iterate as $v_{k} = (y_{k+1} - y_{k})/\sqrt{s}$, the phase-space representation of~\eqref{eqn: nag-y-modify} is written down as
\begin{equation}
\label{eqn: phase-nag-y}
\left\{
\begin{aligned}
           & y_{k} - y_{k-1} = \sqrt{s} v_{k-1},  \\
           & v_{k} - v_{k-1} = - 2\sqrt{\mu s}v_{k} - \sqrt{s}(\nabla f(y_{k}) - f(y_{k-1})) - \sqrt{s}\nabla f(y_{k}),
\end{aligned}
\right.
\end{equation}
with any initial $y_0 = x_0 \in \mathbb{R}^d$ and $v_0 = - \frac{\sqrt{s}}{1+2\sqrt{\mu s}} \cdot \nabla f(x_0)$. Subsequently, we show how to construct a Lyapunov function in the principled way.

\subsection{The principled way of constructing a Lyapunov function} 
\label{subsec: principle-lyapunov}
Corresponding to the phase-space representation~\eqref{eqn: phase-nag-y}, a Lyapunov function is mainly composed of potential, kinetic, and mixed energy, with the mathematical expressions shown in~\Cref{tab: Lyapunov}.
\begin{table}[htbp!]
\centering
\begin{tabular}{l|l}
   \toprule
   Potential Energy                    & $f(y_k) - f(x^\star)$       \\ 
   \midrule
   Kinetic Energy                        & $\frac{1}{2}\|v_{k}\|^2$     \\   
   \midrule                       
   Mixed Energy                          & $\frac{1}{2}\|v_{k} + 2\sqrt{\mu}\left(y_{k+1} - x^\star\right) + \sqrt{s}\nabla f(y_{k})\|^2$          \\
   \bottomrule
\end{tabular}
\caption{The mathematical expression of main ingredients composing a Lyapunov function.}
\label{tab: Lyapunov}
\end{table}

%
%
%

\begin{itemize}
\item[\textbf{(1)}] First, we consider the mixed energy. The core crux of calculating the iterative difference is
\begin{multline*}
     \big[v_{k+1} + 2\sqrt{\mu}(y_{k+2} - x^\star) + \sqrt{s}\nabla f(y_{k+1})\big] - \left[v_{k} + 2\sqrt{\mu}(y_{k+1} - x^\star) + \sqrt{s}\nabla f(y_{k})\right] \\
   =  (v_{k+1} - v_{k}) + 2\sqrt{\mu} (y_{k+2} - y_{k+1}) + \sqrt{s}(\nabla f(y_{k}) - f(y_{k-1})) =  -\sqrt{s}\nabla f(y_{k+1}),
\end{multline*}
where the last equality follows the phase-space representation~\eqref{eqn: phase-nag-y}. Hence, the iterative difference of the mixed energy obeys
\begin{align}
\mathbf{I} = & \frac{1}{2}\|v_{k+1} + 2\sqrt{\mu}\left(y_{k+2} - x^\star\right) + \sqrt{s}\nabla f(y_{k+1})\|^2 - \frac{1}{2}\|v_{k} + 2\sqrt{\mu}\left(y_{k+1} - x^\star\right) + \sqrt{s}\nabla f(y_{k})\|^2  \nonumber\\
= & \left\langle- \sqrt{s}\nabla f(y_{k+1}), v_{k} + 2\sqrt{\mu}(y_{k+1} - x^\ast) + \sqrt{s}\nabla f(y_{k})\right\rangle  + \frac{s}{2}\|\nabla f(y_{k+1})\|^{2}\nonumber\\
    = & -\underbrace{\sqrt{s}\langle\nabla f(y_{k+1}), v_{k}\rangle}_{\textbf{I}_{1}}-\underbrace{ 2\sqrt{\mu s}\langle\nabla f(y_{k+1}), y_{k+1} - x^\star\rangle}_{\textbf{I}_{2}} \nonumber\\
    & -  s\langle\nabla f(y_{k+1}), \nabla f(y_{k})\rangle + \frac{s}{2}\|\nabla f(y_{k+1})\|^{2}. \label{eqn:eqn: mix-gc1}
\end{align}
Here, we adopt the forward calculation in the second equality of~\eqref{eqn:eqn: mix-gc1}, which is the key difference from~\citep{shi2021understanding}. 

\item[\textbf{(2)}] Then, we consider the kinetic energy. Consistently, different from~\citep{shi2021understanding}, we also adopt the forward calculation for the iterative difference as
\begin{align}
    \mathbf{II} = & \langle v_{k+1} - v_{k}, v_{k}\rangle + \frac{1}{2}\|v_{k+1} - v_{k}\|^2 \nonumber \\
  = & - \underbrace{\sqrt{s}\left\langle\nabla f(y_{k+1}), v_{k}\right\rangle}_{\textbf{II}_{1}}  - \underbrace{\sqrt{s}\left\langle\nabla f(y_{k+1}) - \nabla f(y_{k}), v_{k}\right\rangle}_{\textbf{II}_{2}} \nonumber \\& - 2\sqrt{\mu s}\left\langle v_{k+1}, v_{k}\right\rangle+ \frac{1}{2}\left\|v_{k+1} - v_{k}\right\|^2. \label{eqn: kin-gc1}
\end{align}


We merge the last two terms of~\eqref{eqn: kin-gc1} together and calculate it as
\begin{align*}
- 2\sqrt{\mu s}\langle v_{k+1}, v_{k}\rangle + \frac{1}{2}\|v_{k+1} - v_{k}\|^2 & = \frac{1}{2}\|v_{k+1}\|^{2} + \frac{1}{2}\|v_{k}\|^{2} - \left(1+2\sqrt{\mu s}\right)\langle v_{k+1}, v_{k}\rangle \\
    & =  \frac{1}{2}\left\|(1+2\sqrt{\mu s})v_{k+1} - v_{k}\right\|^{2} - 2\sqrt{\mu s}(1+\sqrt{\mu s})\|v_{k+1}\|^{2} \\
    & =  \frac{s}{2}\left\|2\nabla f(y_{k+1}) - \nabla f(y_{k}) \right\|^{2} -\underbrace{2\sqrt{\mu s}(1+\sqrt{\mu s})\|v_{k+1}\|^{2}}_{\textbf{II}_{3}},
\end{align*}
where the last equality follows the second scheme of~\eqref{eqn: phase-nag-y}. Hence, the iterative difference~\eqref{eqn: kin-gc1} can be further reformulated as
\begin{align}
    \mathbf{II}  = & - \underbrace{\sqrt{s}\left\langle\nabla f(y_{k+1}), v_{k}\right\rangle}_{\textbf{II}_{1}}  - \underbrace{\sqrt{s}\left\langle\nabla f(y_{k+1}) - \nabla f(y_{k}), v_{k}\right\rangle}_{\textbf{II}_{2}} \nonumber \\& + \frac{s}{2}\left\|2\nabla f(y_{k+1}) - \nabla f(y_{k}) \right\|^{2} - \underbrace{ 2\sqrt{\mu s}(1+\sqrt{\mu s})\|v_{k+1}\|^{2}}_{\textbf{II}_{3}}. \label{eqn: ken-gc2}
\end{align}

\item[\textbf{(3)}] Finally, we calculate the iterative difference of potential energy as
\begin{align}
\mathbf{III}  = f(y_{k+1}) - f(y_{k}) &\leq  \langle\nabla f(y_{k+1}), y_{k+1} - y_{k}\rangle - \frac{1}{2L}\|\nabla f(y_{k+1})-\nabla f(y_{k})\|^{2} \nonumber \\
    &\leq  \underbrace{\sqrt{s}\langle\nabla f(y_{k+1}), v_{k}\rangle}_{\textbf{III}_{1}} -  \underbrace{\frac{1}{2L}\|\nabla f(y_{k+1})-\nabla f(y_{k})\|^{2}}_{\textbf{III}_{2}}. \label{eqn: pot-gc}
\end{align}

\end{itemize}

Taking a glance at the three inequalities of iterative differences~\eqref{eqn:eqn: mix-gc1},~\eqref{eqn: kin-gc1}, and~\eqref{eqn: pot-gc}, we can find $\mathbf{I}_1 = \mathbf{II}_1 = \mathbf{III}_1$. Checking the signs of them further, the three terms should satisfy $ \mathbf{III}_1 = \alpha\mathbf{I}_1 + \beta \mathbf{II}_1$ with $\alpha + \beta = 1$, which is consistent with the general setting in physics that the coefficient of total kinetic energy $\|v_k\|^2$ is $1/2$. Moreover, the straightforward process above agrees with the continuous case~\citep[Theorem 1]{shi2021understanding} and extremely simplifies the calculation for the discrete algorithm~\citep[Theorem 3]{shi2021understanding}.  Throughout the paper, we continue to use the setting of coefficients $\alpha = \beta = 1/2$ in~\citep{shi2021understanding}.\footnote{Both the coefficients, $\alpha$ and $\beta$, may not be set to $1/2$, which can be changed with the only requirement $\alpha + \beta = 1$} And then, the Lyapunov function is constructed as 
\begin{equation}
\label{eqn: lyp-gc}
\mathcal{E}(k) =    f(y_{k}) - f(x^\star) + \frac{1}{4}\|v_{k}\|^2 + \frac{1}{4}\left\|v_{k} + 2\sqrt{\mu}\left(y_{k+1} - x^\star\right) + \sqrt{s}\nabla f(y_{k})\right\|^2  - \underbrace{\frac{s}{2}\|\nabla f(y_{k})\|^{2}}_{\textbf{additional term}}.
\end{equation}
Moreover, we will find how the additional term is added in the following part of this section.


\subsection{Enlarge step size to $s = 1/L$} 
\label{subsec: step-size-L}
With the Lyapunov function~\eqref{eqn: lyp-gc}, we enlarge the step size to $s = 1/L$ with the following theorem.

\begin{theorem}
\label{thm: rate-gc}
Let $f\in\mathcal{S}_{\mu,L}^{1}$. For any step size $0 < s\leq\frac{1}{L}$, the iterative sequence $\{x_{k}\}_{k=0}^{\infty}$ generated by the  modified~\texttt{NAG}~\eqref{eqn: nag-y-modify} obeys
\begin{equation}
\label{eqn: rate-gc}
f(x_k) - f(x^\star)  \leq  \frac{ 2\left(f(x_0) - f(x^\star) + \mu \|x_0 - x^\star\|^2\right)}{\left(1 + \frac{\sqrt{\mu s}}{4} \right)^{k}},
\end{equation}
for any $k \geq 0$.
\end{theorem}


\begin{proof}[Proof of~\Cref{thm: rate-gc}]
To make the proof clear, we present it in three steps. 

\begin{itemize}

\item[\textbf{(1)}] Based on~\eqref{eqn:eqn: mix-gc1} and~\eqref{eqn: ken-gc2}, we can merge two terms as
\begin{equation}
\label{eqn: total-gc-2}
- \frac{1}{2}\mathbf{I}_{2} - \frac{1}{2}\mathbf{II}_{3} =  - \sqrt{\mu s}\left[(1+\sqrt{\mu s})\|v_{k+1}\|^{2} + \langle\nabla f(y_{k+1}), y_{k+1} - x^\star\rangle\right].
\end{equation}
With the basic inequality of $L$-smooth functions, we can take an estimate as
\begin{equation}
\label{eqn: total-gc-3}
- \frac{1}{2}\mathbf{II}_{2} = -\frac{\sqrt{s}}{2}\langle\nabla f(y_{k+1})-\nabla f(y_{k}), v_{k}\rangle\leq -\frac{1}{2L}\|\nabla f(y_{k+1}) - \nabla f(y_{k})\|^{2}.
\end{equation}
Then, we summarize the reminders together except for $\mathbf{III}_2$ as
 \begin{multline}
   \frac{s}{4}\left\|2\nabla f(y_{k+1}) - \nabla f(y_{k}) \right\|^{2} - \frac{s}{2}\langle\nabla f(y_{k+1}), \nabla f(y_{k})\rangle  + \frac{s}{4}\|\nabla f(y_{k+1})\|^{2}  \\
=   \frac{5s}{4}\|\nabla f(y_{k+1})\|^{2} - \frac{3s}{2}\langle\nabla f(y_{k+1}), \nabla f(y_{k})\rangle + \frac{s}{4}\|\nabla f(y_{k})\|^{2}; \label{eqn: total-gc-4}
\end{multline}
and with~\eqref{eqn: total-gc-3}, we collect the like terms as
\begin{equation}
\label{eqn: total-gc-5}
- \frac{1}{2}\mathbf{II}_{2} - \mathbf{III}_2 \leq -\frac{1}{L}\|\nabla f(y_{k+1}) - \nabla f(y_{k})\|^{2}.
\end{equation}
To make the term~\eqref{eqn: total-gc-4} likely as~\eqref{eqn: total-gc-5} for collecting, we naturally add a term 
\begin{equation}
\label{eqn: add-diff}
- \frac{s}{2}\left(\|\nabla f(y_{k+1})\|^{2} - \|\nabla f(y_{k})\|^{2} \right),
\end{equation}
which corresponds to the additional term in the Lyapunov function~\eqref{eqn: lyp-gc}. Finally, using~\eqref{eqn: total-gc-2},~\eqref{eqn: total-gc-4},~\eqref{eqn: total-gc-5} and~\eqref{eqn: add-diff}, we obtain that the iterative difference of the Lyapunov function satisfies 
\begin{align}
\mathcal{E}(k+1) - \mathcal{E}(k) \leq & - \sqrt{\mu s}\left[(1+\sqrt{\mu s})\|v_{k+1}\|^{2} + \langle\nabla f(y_{k+1}), y_{k+1} - x^\star\rangle\right] \nonumber \\
                                       & - \left( \frac{1}{L} - \frac{3s}{4}\right)\|\nabla f(y_{k+1}) - \nabla f(y_{k})\|^{2}      \label{eqn: total-gc-6}.
\end{align}
Obviously, when the step size satisfies $s \leq 1/L$, the second line of~\eqref{eqn: total-gc-6} is no more than zero.

\item[\textbf{(2)}]  With the phase-space representation~\eqref{eqn: phase-nag-y}, we can estimate the mixed energy by Cauchy-Schwarz inequality as
\begin{align*}
      \left\|v_{k} + 2\sqrt{\mu}\left(y_{k+1} - x^\star\right) + \sqrt{s}\nabla f(y_{k})\right\|^2  &  =    \left\|(1+2\sqrt{\mu s})v_{k} + 2\sqrt{\mu}\left(y_{k} - x^\star\right) + \sqrt{s}\nabla f(y_{k})\right\|^2 \\
&\leq  3(1+2\sqrt{\mu s})^2 \left\|v_{k}\right\|^2 + 12\mu \|y_{k} - x^\star\|^2 + 3s\left\|\nabla f(y_{k})\right\|^2.
\end{align*}
Furthermore, the Lyapunov function~\eqref{eqn: lyp-gc} can be estimated as
\begin{equation}
\label{eqn: lyapunov-gc-estimate}
\mathcal{E}(k) \leq   f(y_{k}) - f(x^\star) +\frac{s\|\nabla f(y_{k})\|^{2}}{4} + (1 + 3\sqrt{\mu s} + 3\mu s)\|v_k\|^2 +3\mu \|y_k - x^\star\|^2               
\end{equation}
With the two following inequalities for $f \in \mathcal{S}_{\mu,L}^{1}$,
\begin{align*}
\langle\nabla f(y_{k+1}), y_{k+1} - x^\star\rangle \geq & \mu\|y_{k+1} - x^\star\|^2, \\
\langle\nabla f(y_{k+1}), y_{k+1} - x^\star\rangle \geq &f(y_{k+1}) - f(x^\star) + \frac{1}{2L}\|\nabla f(y_{k+1})\|^{2},
\end{align*}
the iterative difference~\eqref{eqn: total-gc-6} can be estimated further as
\begin{align}
&\mathcal{E}(k+1) - \mathcal{E}(k)\leq  \nonumber \\
& - \sqrt{\mu s}\left[ \frac{f(y_{k+1}) - f(x^\star)}{4} + \frac{\|\nabla f(y_{k+1})\|^{2}}{8L} + (1+\sqrt{\mu s})\|v_{k+1}\|^{2} + \frac{3\mu\|y_{k+1} - x^\star\|^2}{4}\right]. \label{eqn: total-gc-7}
\end{align}
 Comparing the coefficients of the corresponding terms in~\eqref{eqn: total-gc-7} and~\eqref{eqn: lyapunov-gc-estimate} for $\mathcal{E}(k+1)$,  we conclude that the iterative difference of the discrete Lyapunov function as
\begin{align}
\mathcal{E}(k+1) - \mathcal{E}(k)& \leq - \sqrt{\mu s} \min\left\{ \frac14, \frac{1}{2sL}, \frac{1+\sqrt{\mu s}}{1 + 3\sqrt{\mu s} + 3\mu s}, \frac14 \right\} \mathcal{E}(k+1) \nonumber \\
                                 & \leq - \sqrt{\mu s} \min\left\{ \frac14, \frac{1}{2sL}, \frac{1}{(1 + \sqrt{\mu s})^2}, \frac14 \right\} \mathcal{E}(k+1) \nonumber \\
                                 & \leq - \frac{\sqrt{\mu s}}{4} \mathcal{E}(k+1), \label{eqn: total-gc-8}
\end{align}
where the last inequality follows from $0< s \leq 1/L$.

\item[\textbf{(3)}] It follows from the Lyapunov function~\eqref{eqn: lyp-gc} and the total iterative difference~\eqref{eqn: total-gc-8} that the following inequality holds
\[
f(y_k) - f(x^\star) - \frac{s}{2} \|\nabla f(y_k)\|^2 \leq \mathcal{E}(k) \leq \frac{\mathcal{E}(0)}{\left( 1 + \frac{\sqrt{\mu s}}{4} \right)^{k}}.
\]
With the $L$-smooth condition, we can obtain the following inequality by the gradient step in~\texttt{NAG}~\eqref{eqn: nag}, $x_{k+1} = y_{k} - s\nabla f(y_k)$, as
\[
f(x_{k+1}) - f(x^\star) \leq f(y_k) - f(x^\star) - \frac{s}{2} \|\nabla f(y_k)\|^2 \leq \frac{\mathcal{E}(0)}{\left( 1 + \frac{\sqrt{\mu s}}{4} \right)^{k}} \leq \frac{2\mathcal{E}(0)}{\left( 1 + \frac{\sqrt{\mu s}}{4} \right)^{k+1}}.
\]
Putting the initial conditions, $y_0 = x_0$ and $v_0 = - \frac{\sqrt{s} \nabla f(x_0)}{1+2\sqrt{\mu s}}$, into $\mathcal{E}(0)$, we complete the proof. 
\end{itemize}
\end{proof}

%


\paragraph{Remark}
Two imperfect points appear in the proof above from the modified gradient-correction scheme~\eqref{eqn: nag-y-modify}. 
\begin{itemize}
\item[\textbf{(1)}] In the modified~\texttt{NAG}~\eqref{eqn: nag-y-modify}, we use $1+2\sqrt{\mu s}$ instead of $1$ for the coefficient of gradient term to enlarge the step size to $s = 1/L$. Otherwise, the inequality~\eqref{eqn: total-gc-6} will degenerate to the case appearing in~\citep[Proof of Lemma 3]{shi2021understanding}, that is, if the step size is set to $s=1/L$,
\[
\left\langle \nabla f(x_{k+1}), x_{k+1} - x^\star \right\rangle - s\|\nabla f(x_{k+1})\|^2 \geq 0, 
\]
where equality works for the worst case so that it is not possible to make the corresponding terms proportional.
\item[\textbf{(2)}] In the first step of the proof, we can find the additional term is necessary. Otherwise, we cannot guarantee the iterative difference is no more than zero. In addition, we use the inequality of gradient step~\eqref{eqn: nag} in the third step. Hence, the convergence rate is obtained only by the objective value $f(x_{k+1})$ since the objective values, $f(x_{k+1})$ and $f(y_{k})$, only satisfies one-way inequality. 
\end{itemize}

\section{The proof from the implicit-velocity scheme}
\label{sec: implicit}

In this section, we investigate the iterative behavior from the implicit scheme~\eqref{eqn: nag-x}. To analyze the discrete~\texttt{NAG}, an implicit-velocity high-resolution differential equation with its convergence rate is first provided as guidance. Then, we demonstrate a concise proof of the convergence rate from the implicit-velocity scheme~\eqref{eqn: nag-x}.

\subsection{The implicit-velocity high-resolution differential equation}
\label{subsec: iv-hode}

Similar to~\citep{chen2022gradient} for the convex function, plugging the high-resolution Taylor expansion into~\eqref{eqn: nag-x} and taking the $O(\sqrt{s})$-approximation, we can obtain an implicit-velocity high-resolution differential equation as 
\begin{equation}
\label{eqn: iv-high-ode-original}
(1+\sqrt{\mu s})\ddot{X} + 2\sqrt{\mu}\dot{X} +  (1 + 2\sqrt{\mu s})\nabla f\bigg(X + \underbrace{\frac{\sqrt{s}\dot{X}}{1+2\sqrt{\mu s}}}_{\textbf{implicit velocity}} \bigg) = 0,
\end{equation}
with any initial $X(0)=x_0 \in \mathbb{R}^d$ and $\dot{X}(0)=0$. Since the continuous case is used only as a guide, we simplify~\eqref{eqn: iv-high-ode-original} to show the derivation clearly by ignoring the coefficients except for the implicit velocity as
\begin{equation}
\label{eqn: iv-high-ode}\ddot{X} + 2\sqrt{\mu}\dot{X} +  \nabla f\left(X + \frac{\sqrt{s}\dot{X}}{1+2\sqrt{\mu s}} \right) = 0.
\end{equation}
with the same initial values of~\eqref{eqn: iv-high-ode-original}. Then, we show the principled way to construct a Lyapunov function for~\eqref{eqn: iv-high-ode} as hold.

\begin{itemize}
\item[\textbf{(1)}] With~\eqref{eqn: iv-high-ode}, we first calculate the mixed energy's time derivative as
\begin{align*}
\mathbf{I} = & \frac{\mathrm{d}}{\mathrm{d}t}\left(\frac12\|\dot{X}+2\sqrt{\mu}(X  - x^\star)\|^2 \right) \\
          = &\big\langle \ddot{X} + 2\sqrt{\mu}\dot{X}, \dot{X}+2\sqrt{\mu}(X -x^\star)\big\rangle\\
= &   \left\langle - \nabla f\left(X + \frac{\sqrt{s}\dot{X}}{1+2\sqrt{\mu s}} \right), \dot{X}+2\sqrt{\mu}(X -x^\star)\right\rangle \\
= & -\underbrace{\frac{\left\langle  \nabla f \left(X + \frac{\sqrt{s}\dot{X}}{1+2\sqrt{\mu s}} \right), \dot{X}\right\rangle}{1 + 2\sqrt{\mu s}}}_{\textbf{I}_{1}} - 2\sqrt{\mu} \left\langle \nabla f\left(X + \frac{\sqrt{s} \dot{X}}{1+2\sqrt{\mu s}}\right),  X + \frac{\sqrt{s}\dot{X}}{1+2\sqrt{\mu s}} - x^\star \right\rangle.
\end{align*}

\item[\textbf{(2)}] Then, for the kinetic energy, the time derivative along with~\eqref{eqn: iv-high-ode} is derived as
\begin{align*}
\mathbf{II} = \frac{\mathrm{d}}{\mathrm{d}t}\left( \frac12\|\dot{X}\|^2\right) = &\big\langle\ddot{X}, \dot{X}\big\rangle \\= &  \left\langle -2\sqrt{\mu}\dot{X} - (1+2\sqrt{\mu s}) \nabla f\left(X + \frac{\sqrt{s}\dot{X}}{1+2\sqrt{\mu s}}\right),\dot{X}\right\rangle \\
= & - \underbrace{\left(1+2\sqrt{\mu s}\right)\left\langle\dot{X},  \nabla f\left(X + \frac{\sqrt{s}\dot{X}}{1+2\sqrt{\mu s}}\right)\right\rangle}_{\textbf{II}_{1}}  - 2\sqrt{\mu}\|\dot{X}\|^{2}.
\end{align*}

\item[\textbf{(3)}] Finally, we calculate the time derivative of potential energy with~\eqref{eqn: iv-high-ode} as 
\begin{align*}
\mathbf{III}= \frac{\mathrm{d}}{\mathrm{d}t}\big( f(X) - f(x^\star) \big) = & \left\langle\nabla f\left(X + \frac{\sqrt{s}\dot{X}}{1+2\sqrt{\mu s}}\right), \dot{X} + \frac{\sqrt{s}\ddot{X}}{1+2\sqrt{\mu s}} \right\rangle\\
=&\left\langle\nabla f\left(X + \frac{\sqrt{s}\dot{X}}{1+2\sqrt{\mu s}}\right),  \frac{\dot{X} - \sqrt{s}  \nabla f\left(X + \frac{\sqrt{s}\dot{X}}{1+2\sqrt{\mu s}} \right)}{(1+2\sqrt{\mu s})}\right\rangle\\
= &\underbrace{ \frac{\left\langle  \nabla f\left(X + \frac{\sqrt{s}\dot{X}}{1+2\sqrt{\mu s}}\right), \dot{X}\right\rangle}{1+2\sqrt{\mu s}}}_{\textbf{III}_{1}} -  \left( \frac{\sqrt{ s}}{1 + 2\sqrt{\mu s} } \right)\left\| \nabla f\left(X + \frac{\sqrt{s}\dot{X}}{1+2\sqrt{\mu s}}\right)\right\|^2.
\end{align*}
\end{itemize}

%
Similarly, by comparing the coefficients, we find that there exists  the following relation among the three labeled terms above as
\[
\mathbf{I}_1 = \frac{\mathbf{II}_1}{(1+2\sqrt{\mu s})^2} = \mathbf{III}_1.
\]
Then, we construct the Lyapunov function corresponding to~\eqref{eqn: iv-high-ode} as
\begin{equation}
\mathcal{E}(t) = f\left(X + \frac{\sqrt{s}\dot{X}}{1+2\sqrt{\mu s}}\right) - f(x^\star) + \frac{\|\dot{X}\|^{2}}{4(1+2\sqrt{\mu s})^2}   
                                                                      + \frac{\|\dot{X}+2\sqrt{\mu}(X  - x^\star)\|^2}{4}. \label{eqn: lypaunov-ode-iv}
\end{equation}
With the Laypunov function~\eqref{eqn: lypaunov-ode-iv}, we formalize the convergence rate of objective values along the implicit-velocity high-resolution differential euqation~\eqref{eqn: iv-high-ode} with the following theorem.

\begin{theorem}
\label{thm:continuous-implcit-velocity}
Let $f\in \mathcal{S}^1_{\mu,L}$, then the solution $X = X(t)$ to the implicit-velocity high-resolution differential equation~\eqref{eqn: iv-high-ode} satisfies
\begin{equation}
\label{eqn: continuous}
f\left(X + \frac{\sqrt{s}\dot{X}}{1+2\sqrt{\mu s}} \right) - f(x^\star) \leq \frac{f(x_0) - f(x^\star) + \mu\|x_{0}-x^\ast\|^{2}}{2}\text{e}^{-\frac{\sqrt{\mu}t}{4}},
\end{equation}
for any $t \geq 0$.
\end{theorem}

\begin{proof}[Proof of~\Cref{thm:continuous-implcit-velocity}]
First, using Cauchy-Schwarz inequality, we estimate the Lyapunov function~\eqref{eqn: lypaunov-ode-iv} as
\begin{equation}
\label{eqn: lypaunov-ode-iv-estimate}
\mathcal{E}(t) \leq f\left(X + \frac{\sqrt{s}\dot{X}}{1+2\sqrt{\mu s}}\right) - f(x^\star) + \frac{3\|\dot{X}\|^{2}}{4(1+2\sqrt{\mu s})^2}   
                                                                      + 2\mu\|X  - x^\star\|^2.
\end{equation}
Then, together with  $\mathbf{I}$, $\mathbf{II}$ and $\mathbf{III}$ above, we calculate the time derivative of $\mathcal{E}(t)$ in~\eqref{eqn: lypaunov-ode-iv} as
\begin{align}
\frac{d\mathcal{E}(t)}{dt} &\leq  - \sqrt{\mu} \left( \frac{\|\dot{X}\|^{2}}{(1 + 2\sqrt{\mu s})^2} + \left\langle \nabla f\left(X + \frac{\sqrt{s} \dot{X}}{1+2\sqrt{\mu s}}\right),  X + \frac{\sqrt{s}\dot{X}}{1+2\sqrt{\mu s}} - x^\star \right\rangle \right) \nonumber \\
                        &\leq  - \sqrt{\mu} \left( f\left(X + \frac{\sqrt{s} \dot{X}}{1+2\sqrt{\mu s}} \right) - f(x^\star) + \frac{\|\dot{X}\|^{2}}{(1 + 2\sqrt{\mu s})^2} + \frac{\mu}{2} \left\|  X + \frac{\sqrt{s} \dot{X}}{1+2\sqrt{\mu s}} - x^\star\right\|^2 \right)        \label{eqn: im-continous-final}
\end{align}
where the last inequality follows that the basic inequality of $f \in \mathcal{S}_{\mu,L}^{1}$, that is,
\begin{multline*}
     \left\langle \nabla f\left(X + \frac{\sqrt{s} \dot{X}}{1+2\sqrt{\mu s}}\right),  X + \frac{\sqrt{s}\dot{X}}{1+2\sqrt{\mu s}} - x^\star \right\rangle \\
\geq  f\left(X + \frac{\sqrt{s} \dot{X}}{1+2\sqrt{\mu s}} \right) - f(x^\star) + \frac{\mu}{2} \left\|  X + \frac{\sqrt{s} \dot{X}}{1+2\sqrt{\mu s}} - x^\star\right\|^2.
\end{multline*}
 Comparing the coefficients of the corresponding terms in~\eqref{eqn: im-continous-final} and~\eqref{eqn: lypaunov-ode-iv-estimate} for $\mathcal{E}(t)$,  we find that the time derivative satisfies
\[
\frac{d\mathcal{E}(t)}{dt} \leq  - \sqrt{\mu} \min\left\{1, \frac13, \frac14 \right\} \mathcal{E}(t) = - \frac{\sqrt{\mu}}{4}\mathcal{E}(t). 
\]
Plugging the initial conditions, $X(0) = x_0 $ and $ \dot{X}(0) = 0$, into $\mathcal{E}(0)$, we complete the proof with the basic $L$-smooth inequality. 
\end{proof}

\subsection{The implicit-velocity scheme}
\label{subsec: iv-scheme}

Similar with~\citep{chen2022gradient},  we also takes the velocity iterates here as $v_{k} = (x_k - x_{k-1})/\sqrt{s}$. Then, the phase-space representation of the implicit-velocity scheme~\eqref{eqn: nag-x} is

\begin{equation}
\label{eqn: phase-nag-x}
\left\{
\begin{aligned}
           & x_{k+1} - x_{k} = \sqrt{s}v_{k+1},  \\
           & v_{k+1} - v_{k} = - \frac{2\sqrt{\mu s}v_{k}}{1+2\sqrt{\mu s}} - \sqrt{s}\nabla f\left(x_{k} + \frac{\sqrt{s}v_{k}}{1+2\sqrt{\mu s}}\right),
\end{aligned}
\right.
\end{equation}
with any initial $x_0 \in \mathbb{R}^d$ and $v_0 = 0$. With the phase-space representation~\eqref{eqn: phase-nag-x}, the second scheme of~\texttt{NAG}~\eqref{eqn: nag} can be reformulated as
 \begin{equation}
\label{eqn: nag-2-iv}
    y_{k+1} = x_{k+1} + \frac{\sqrt{s}v_{k+1}}{1+2\sqrt{\mu s}}.
\end{equation}
Next, we demonstrate how to construct a Lyapunov function by the principled way.


\paragraph{The principled way of constructing a Lyapunov function}

\begin{itemize}
\item[\textbf{(1)}] First, we consider the mixed energy. The core crus of calculating the iterative difference is 
\[
\left[ v_{k+2}  + 2\sqrt{\mu}(x_{k+2} - x^\star) \right] -  \left[ v_{k+1}  + 2\sqrt{\mu}(x_{k+1} - x^\star) \right] =  - \sqrt{s} \nabla f(y_{k+1}),
\]
which directly follows the second scheme of~\eqref{eqn: phase-nag-x}. Hence, the iterative difference of the mixed energy obeys
\begin{align*}
 \mathbf{I} =  & \frac{1}{2}\left\| v_{k+2}  + 2\sqrt{\mu}(x_{k+2} - x^\star) \right\|^2 -  \frac{1}{2}\left\| v_{k+1}  + 2\sqrt{\mu}(x_{k+1} - x^\star) \right\|^2 \\
=    &  \left\langle  - \sqrt{s} \nabla f(y_{k+1}),  v_{k+1}  + 2\sqrt{\mu}(x_{k+1} - x^\star) \right\rangle + \frac{s}{2}\| \nabla f(y_{k+1})\|^2  \\
=    &  \left\langle  - \sqrt{s} \nabla f(y_{k+1}),   \frac{v_{k+1}}{1+2\sqrt{\mu s}}  + 2\sqrt{\mu}(y_{k+1} - x^\star) \right\rangle + \frac{s}{2}\| \nabla f(y_{k+1})\|^2,
\end{align*}
where the last inequality follows~\eqref{eqn: nag-2-iv}. Furthermore, we reformulate $\mathbf{I}$ by the second scheme of~\eqref{eqn: phase-nag-x} as
\begin{equation}
\label{eqn: iv-1}
\mathbf{I} = - \underbrace{ \sqrt{s} \left\langle  \nabla f(y_{k+1}), v_{k+2}\right\rangle}_{\mathbf{I}_1}  - 2\sqrt{\mu s}  \left\langle \nabla f(y_{k+1}), y_{k+1} - x^\star \right\rangle - \frac{s}{2}   \|\nabla f(y_{k+1})\|^2. 
\end{equation}

\item[\textbf{(2)}] Then, we consider the kinetic energy. In the second line of the phase-space representation~\eqref{eqn: phase-nag-x}, there are only two terms on the right-hand side, which is the key difference of the implicit-velocity scheme from the gradient-correction scheme~\eqref{eqn: phase-nag-y}.  The detailed calculation of iterative difference is shown as
\begin{align}
\mathbf{II} & =  \frac12 \|v_{k+2}\|^2 - \frac12 \|v_{k+1}\|^2 \nonumber \\
            & =  \frac{1}{2} \|v_{k+2}\|^2 - \frac{(1+ 2\sqrt{\mu s})^2}{2} \|v_{k+2} + \sqrt{s} \nabla f(y_{k+1})\|^2 \nonumber \\
            & =  - \underbrace{(1+ 2\sqrt{\mu s})^2 \sqrt{s} \langle \nabla f(y_{k+1}), v_{k+2} \rangle}_{\mathbf{II}_1} - 2\sqrt{\mu s}(1 + \sqrt{\mu s}) \|v_{k+2}\|^2 \nonumber \\
            & \mathrel{\phantom{ =  - \underbrace{(1+ 2\sqrt{\mu s})^2 \sqrt{s} \langle \nabla f(y_{k+1}), v_{k+2} \rangle}_{\mathbf{II}_1}}}  - \frac{s(1+ 2\sqrt{\mu s})^2}{2}\| \nabla f(y_{k+1})\|^2 . \label{eqn: iv-2}
\end{align}

\item[\textbf{(3)}] Finally, with the basic $L$-smooth inequality, we calculate the iterative difference of potential energy as
\begin{align}
\mathbf{III}  & = f(y_{k+1}) - f(y_{k}) \nonumber \\
              & \leq \langle \nabla f(y_{k+1}), y_{k+1} - y_{k} \rangle - \frac{1}{2L} \| \nabla f(y_{k+1}) - \nabla f(y_{k}) \|^2 \nonumber \\
              & = \langle \nabla f(y_{k+1}), x_{k+2} - x_{k+1} \rangle + s\langle \nabla f(y_{k+1}), \nabla f(y_{k+1}) - \nabla f(y_{k}) \rangle  - \frac{1}{2L} \| \nabla f(y_{k+1}) - \nabla f(y_{k}) \|^2 \nonumber\\
              & = \underbrace{\sqrt{s} \langle \nabla f(y_{k+1}), v_{k+2} \rangle}_{\mathbf{III}_1} - \frac{1}{2}\left(\frac{1}{L} - s \right) \| \nabla f(y_{k+1}) - \nabla f(y_{k}) \|^2 \nonumber \\
              & \mathrel{\phantom{ = \underbrace{\sqrt{s} \langle \nabla f(y_{k+1}), v_{k+2} \rangle}_{\mathbf{III}_1}}}   + \frac{s}{2} \|\nabla f(y_{k+1})\|^2 - \frac{s}{2} \|\nabla f(y_{k})\|^2, \label{eqn: iv-3}
\end{align}
where the second equality follows from the gradient step of~\texttt{NAG}~\eqref{eqn: nag}. 


\end{itemize}

Consistent with the continuous case, we also find that the three labeled terms above have the following relation as
\[
\mathbf{I}_1 = \frac{\mathbf{II}_1}{(1+2\sqrt{\mu s})^2} = \mathbf{III}_1.
\]
Then, we construct the Lyapunov function corresponding to~\texttt{NAG} as
\begin{equation}
\mathcal{E}(k) = f(y_{k}) - f(x^{\star}) + \frac{\|v_{k+1}\|^2}{4(1+2\sqrt{\mu s})^2} +  \frac{\left\|v_{k+1} +2\sqrt{\mu}(x_{k+1} - x^\star)\right\|^2 }{4}.\label{eqn: lyapunov-iv2}
\end{equation}

We also can obtain the accelerated convergence only using the Lyapunov function~\eqref{eqn: lyapunov-iv2} without the additional term, but the proportional constant will become too small. We will discuss it at the end of this section.


\begin{theorem}
\label{thm: rate-iv}
Let $f\in\mathcal{S}_{\mu,L}^{1}$. For any step size $0 < s \leq 1/L$, the iterates $\{x_{k}\}_{k=0}^{\infty}$ generated by \texttt{NAG}~\eqref{eqn: nag} obeys
\begin{equation}
\label{eqn: convergence-rate}
f(y_{k}) - f(x^\star) \leq \frac{4L\|x_{0}-x^\star\|^{2}}{\left(1+ \frac{\sqrt{\mu s}}{4}\right)^{k}},
\end{equation}
for any $k\geq0$.
\end{theorem}

\begin{proof}[Proof of~\Cref{thm: rate-iv}]
From the implicit-velocity scheme~\eqref{eqn: nag-x}, only two steps are enough to present the proof. 
\begin{itemize}
\item[\textbf{(1)}] With the gradient step of~\texttt{NAG}~\eqref{eqn: nag}, we can estimate the mixed energy by Cauchy-Schwarz inequality as
\begin{align*}
\|v_{k+1} +2\sqrt{\mu}(x_{k+1} - x^\star) \|^2  & =     \left\| v_{k+1} +2\sqrt{\mu}(y_k - s\nabla f(y_k) - x^\star)  \right\|^2 \\
                                                & \leq  2 \|v_{k+1}\|^2 + 8 \mu \|y_k - s\nabla f(y_k) - x^\star\|^2             \\ 
                                                & =     2 \|v_{k+1}\|^2 + 8 \mu \|y_k  - x^\star\|^2 +  8 \mu s^2 \| \nabla f(y_k) \|^2 - 16 \mu s \langle \nabla f(y_k), y_k - x^\star \rangle \\
                                                & \leq  2 \|v_{k+1}\|^2 + 8 \mu \|y_k  - x^\star\|^2,     
\end{align*}
where the last inequality works for the step size $0<s\leq 1/L$. Furthermore, the Lyapunov function~\eqref{eqn: lyapunov-iv2} can be estimated as
\begin{align*}
\mathcal{E}(k) & \leq f(y_{k}) - f(x^{\star}) +   \frac{\left(\frac{3}{4}+\sqrt{\mu s} + \mu s\right)}{(1+2\sqrt{\mu s})^2} \|v_{k+1}\|^2  + 2\mu \|y_k - x^\star\|^2. 
\end{align*}
\item[\textbf{(2)}]
With~\eqref{eqn: iv-1},~\eqref{eqn: iv-2} and~\eqref{eqn: iv-3}, the iterative difference of the Lyapunov function~\eqref{eqn: lyapunov-iv2} can be estimated further as
\begin{multline}
\mathcal{E}(k+1) - \mathcal{E}(k)   \leq  -  \sqrt{\mu s}  \left\langle \nabla f(y_{k+1}), y_{k+1} - x^\star \right\rangle - \frac{(\sqrt{\mu s} + \mu s) \|v_{k+2}\|^2}{(1+2\sqrt{\mu s})^2} \\
                                                      \leq -\sqrt{\mu s} \left[ f(y_{k+1}) - f(x^\star) + \frac{(1 + \sqrt{\mu s} ) \|v_{k+2}\|^2}{(1+2\sqrt{\mu s})^2}  + \frac{\mu \|y_{k+1} - x^\star\|^2 }{2} \right], \label{eqn: iterative-difference-iv-lyp} 
\end{multline}
where the last inequality follows from the inequality of $f \in \mathcal{S}_{\mu,L}^{1}$
\[
\left\langle \nabla f(y_{k+1}), y_{k+1} - x^\star \right\rangle \geq f(y_{k+1}) - f(x^\star) + \frac{\mu}{2} \|y_{k+1} - x^\star\|^2.
\]
Comparing the coefficients of the corresponding terms in~\eqref{eqn: iterative-difference-iv-lyp} and~\eqref{eqn: lyapunov-iv2} for $\mathcal{E}(k+1)$, we obtain that the iterative difference satisfies
\[
\mathcal{E}(k+1) - \mathcal{E}(k)  \leq -\sqrt{\mu s} \min\left\{ 1, \frac{1+\sqrt{\mu s}}{\frac{3}{4}+\sqrt{\mu s} + \mu s}, \frac14 \right\} \mathcal{E}(k+1) = - \frac{\sqrt{\mu s}}{4}  \mathcal{E}(k+1).
\]
Plugging the initial conditions $x_0 \in \mathbb{R}^d$ and $v_1 = 0$ into $\mathcal{E}(0)$, we complete the proof.
\end{itemize}
\end{proof}

If we set the first velocity iterate as $v_1 = 2\sqrt{\mu s} \nabla f(y_{k})$, then the estimate of the initial Lyapunov function can be simplified. Then, we conclude this section with the following corollary. 
\begin{corollary}
\label{thm: rate-iv-x}
Let $f\in\mathcal{S}_{\mu,L}^{1}$. If the step size  is set $s = 1/L$, the iterates $\{x_{k}\}_{k=0}^{\infty}$ generated by \texttt{NAG}~\eqref{eqn: nag} obeys
\begin{equation}
\label{eqn: con-rate}
f(x_{k}) - f(x^\star) \leq \frac{2(f(x_0) - f(x^\star)) + \mu\|x_{0}-x^\star\|^{2}}{\left(1+ \frac{1}{4}\sqrt{\frac{\mu}{L}} \right)^{k}},
\end{equation}
for any $k\geq0$.
\end{corollary}

%

%

\section{Discussion and conclusion}
\label{sec: discussion}

In this study, we construct two Lyapunov functions~\eqref{eqn: lyp-gc} and~\eqref{eqn: lyapunov-iv2} for the gradient-correction scheme and the implicit-velocity scheme, respectively. Indeed, we find that both the potential and the mixed energy are identical in the Lyapunov functions~\eqref{eqn: lyp-gc} and~\eqref{eqn: lyapunov-iv2} with the only difference located in the kinetic energy, which is also the essential difference from the convex case. Moreover, the computation from the implicit-velocity scheme is superior to the gradient-correction scheme for the $\mu$-strongly convex function. Collectively, the high-resolution differential equation framework from the implicit-velocity scheme is tailor-made for the~\texttt{NAG}s. 

Based on the high-resolution differential equation framework, we obtain the convergence rate of objective values in both the schemes~\eqref{eqn: rate-gc} and~\eqref{eqn: convergence-rate} is $(1+\sqrt{\mu s}/4)^{-k}$, where the coefficient $1/4$ can be definitely improved. However, whether the coefficient $1/4$ can be enlarged to $1$, consistent with~\citep[Theorem 2.2.2]{nesterov1998introductory}, is still unknown. Furthermore, the lower bound given in~\citep[Theorem 2.1.13]{nesterov1998introductory} comes from an infinite-dimensional example. In other words, for any finite dimension $d$, the lower bound is only valid before the iterative time $k$ is no more than $d$. Indeed, this is similar to the convex case, and there is also no lower bound in infinite iterates for the $\mu$-strongly convex function.

Recall the vanilla gradient descent, which is always accompanied by monotonic convergence. The monotonicity can guarantee that the objective value at the current iteration is always better than the previous ones, so we easily choose the optimal one in practice. However, it has been pointed out in~\citep{danilova2020non} that accelerated first-order methods are not guaranteed to be monotonic. Take a quadratic function $f(y) = \frac12 y^{T}\Lambda y$ with $\Lambda = \text{diag}(\lambda_1, \ldots, \lambda_d)$ as  an example. Plugging it into the modified~\texttt{NAG}~\eqref{eqn: phase-nag-y}, we can obtain  
\[
(1+2\sqrt{\mu s})y_{k+1} - (1 - \lambda_i s + \sqrt{\mu s})y_k + (1 - \lambda_i s)y_{k-1} = 0,
\]
where $i=1,\ldots, d$. With the rule of numerical stability, the corresponding characteristic equation of $y_{k}^{(i)}$ can be written down as
\[
(1+2\sqrt{\mu s})\alpha^2 - (1 - \lambda_i s + \sqrt{\mu s})\alpha + 1 - \lambda_i s = 0,
\]
where $i=1,\ldots, d$. For the special case, the monotonic convergence of~\eqref{eqn: phase-nag-y} is reduced to that the roots of the quadratic equation are real. In other words, the discriminant is no less than zero, that is, $s \geq (\lambda_i - \mu)/\lambda_i^2$. Taking the maximum, we have 
\[
\max_{\lambda_i \in \mathbb{R}} \frac{\lambda_i - \mu}{\lambda_i^2} = \frac{1}{4\mu}.
\]
Then, if we assume that the condition satisfies $\max_{1\leq i \leq d} \lambda_i = L \leq 4\mu$, the modified~\texttt{NAG}~\eqref{eqn: phase-nag-y} always possess the accelerated convergence accompanied by monotonicity when the step size $s$ is chosen in $[1/(4\mu), 1/L]$. As a result, both acceleration and monotonicity can coexist. Hence, an interesting research direction is designing new algorithms to simultaneously accelerate the vanilla gradient descent accompanied by monotonicity.

{\small
\subsection*{Acknowledgments}
We would like to thank Bowen Li for helpful discussions.
\bibliographystyle{abbrvnat}
\bibliography{reference}

\begin{thebibliography}{18}
\providecommand{\natexlab}[1]{#1}
\providecommand{\url}[1]{\texttt{#1}}
\expandafter\ifx\csname urlstyle\endcsname\relax
  \providecommand{\doi}[1]{doi: #1}\else
  \providecommand{\doi}{doi: \begingroup \urlstyle{rm}\Url}\fi

\bibitem[Attouch et~al.(2014)Attouch, Peypouquet, and
  Redont]{attouch2014dynamical}
H.~Attouch, J.~Peypouquet, and P.~Redont.
\newblock A dynamical approach to an inertial forward-backward algorithm for
  convex minimization.
\newblock \emph{SIAM Journal on Optimization}, 24\penalty0 (1):\penalty0
  232--256, 2014.

\bibitem[Attouch et~al.(2020)Attouch, Chbani, Fadili, and
  Riahi]{attouch2020first}
H.~Attouch, Z.~Chbani, J.~Fadili, and H.~Riahi.
\newblock First-order optimization algorithms via inertial systems with
  {H}essian driven damping.
\newblock \emph{Mathematical Programming}, pages 1--43, 2020.

\bibitem[Chen et~al.(2022)Chen, Shi, and Yuan]{chen2022gradient}
S.~Chen, B.~Shi, and Y.-X. Yuan.
\newblock Gradient norm minimization of {N}esterov acceleration: $ o (1/k^{3})
  $.
\newblock \emph{arXiv preprint arXiv:2209.08862}, 2022.

\bibitem[Danilova et~al.(2020)Danilova, Kulakova, and Polyak]{danilova2020non}
M.~Danilova, A.~Kulakova, and B.~Polyak.
\newblock Non-monotone behavior of the heavy ball method.
\newblock In \emph{International Conference on Difference Equations and
  Applications}, pages 213--230. Springer, 2020.

\bibitem[Gelfand and Tsetlin(1961)]{gelfand1962some}
I.~M. Gelfand and M.~L. Tsetlin.
\newblock Principle of the nonlocal search in the systems of automatic
  optimization.
\newblock \emph{Dokl. Akad. Nauk SSSR}, 137\penalty0 (2):\penalty0 295--298,
  1961.

\bibitem[Jordan(2018)]{jordan2018dynamical}
M.~I. Jordan.
\newblock Dynamical, symplectic and stochastic perspectives on gradient-based
  optimization.
\newblock In \emph{Proceedings of the International Congress of Mathematicians:
  Rio de Janeiro 2018}, pages 523--549. World Scientific, 2018.

\bibitem[Li et~al.(2022)Li, Shi, and Yuan]{li2022proximal}
B.~Li, B.~Shi, and Y.-X. Yuan.
\newblock Proximal subgradient norm minimization of {ISTA} and {FISTA}.
\newblock \emph{arXiv preprint arXiv:2211.01610}, 2022.

\bibitem[Luo and Chen(2021)]{luo2021differential}
H.~Luo and L.~Chen.
\newblock From differential equation solvers to accelerated first-order methods
  for convex optimization.
\newblock \emph{Mathematical Programming}, pages 1--47, 2021.

\bibitem[Muehlebach and Jordan(2019)]{muehlebach2019dynamical}
M.~Muehlebach and M.~Jordan.
\newblock A dynamical systems perspective on {N}esterov acceleration.
\newblock In \emph{International Conference on Machine Learning}, pages
  4656--4662. PMLR, 2019.

\bibitem[Nesterov(1998)]{nesterov1998introductory}
Y.~Nesterov.
\newblock \emph{Introductory {L}ectures on {C}onvex {O}ptimization: A {B}asic
  {C}ourse}, volume~87.
\newblock Springer Science \& Business Media, 1998.

\bibitem[Nesterov(2010)]{nesterov2010recent}
Y.~Nesterov.
\newblock Recent advances in structural optimization.
\newblock In \emph{Proceedings of the International Congress of Mathematicians
  2010 (ICM 2010) (In 4 Volumes) Vol. I: Plenary Lectures and Ceremonies Vols.
  II--IV: Invited Lectures}, pages 2964--2978. World Scientific, 2010.

\bibitem[Nesterov(1983)]{nesterov1983method}
Y.~E. Nesterov.
\newblock A method for solving the convex programming problem with convergence
  rate ${O}(1/k^2)$.
\newblock \emph{Doklady {A}kademii {N}auk {SSSR}}, 269:\penalty0 543--547,
  1983.

\bibitem[Polyak(1964)]{polyak1964some}
B.~T. Polyak.
\newblock Some methods of speeding up the convergence of iteration methods.
\newblock \emph{USSR Computational Mathematics and Mathematical Physics},
  4\penalty0 (5):\penalty0 1--17, 1964.

\bibitem[Shi(2021)]{shi2021hyperparameters}
B.~Shi.
\newblock On the hyperparameters in stochastic gradient descent with momentum.
\newblock \emph{arXiv preprint arXiv:2108.03947}, 2021.

\bibitem[Shi et~al.(2022)Shi, Du, Jordan, and Su]{shi2022understanding}
B.~Shi, S.~S. Du, M.~I. Jordan, and W.~J. Su.
\newblock Understanding the acceleration phenomenon via high-resolution
  differential equations.
\newblock \emph{Mathematical Programming}, 195\penalty0 (1):\penalty0 79--148,
  2022.

\bibitem[Wilson et~al.(2021)Wilson, Recht, and Jordan]{wilson2021lyapunov}
A.~C. Wilson, B.~Recht, and M.~I. Jordan.
\newblock A {L}yapunov analysis of accelerated methods in optimization.
\newblock \emph{J. Mach. Learn. Res.}, 22:\penalty0 113--1, 2021.

\bibitem[Xia et~al.(2021)Xia, Suliafu, Ji, Nguyen, Bertozzi, Osher, and
  Wang]{xia2021heavy}
H.~Xia, V.~Suliafu, H.~Ji, T.~Nguyen, A.~Bertozzi, S.~Osher, and B.~Wang.
\newblock Heavy ball neural ordinary differential equations.
\newblock \emph{Advances in Neural Information Processing Systems},
  34:\penalty0 18646--18659, 2021.

\bibitem[Zhang et~al.(2021)Zhang, Orvieto, Daneshmand, Hofmann, and
  Smith]{zhang2021revisiting}
P.~Zhang, A.~Orvieto, H.~Daneshmand, T.~Hofmann, and R.~S. Smith.
\newblock Revisiting the role of euler numerical integration on acceleration
  and stability in convex optimization.
\newblock In \emph{International Conference on Artificial Intelligence and
  Statistics}, pages 3979--3987. PMLR, 2021.

\end{thebibliography}
}
\end{document}